\documentclass{article}
\pdfoutput=1
\usepackage[utf8]{inputenc}
\usepackage[T1]{fontenc}
\usepackage{amsmath,amssymb,amsfonts,amsthm,graphicx,mathrsfs,url}
\usepackage[usenames,dvipsnames]{color}
\usepackage[colorlinks=true,linkcolor=Red,citecolor=Green,urlcolor=Blue]{hyperref}
\usepackage[super]{nth}
\usepackage[open, openlevel=2, depth=3, atend]{bookmark}
\usepackage{physics}
\hypersetup{pdfstartview=XYZ}

\usepackage{bbold}

\def\?[#1]{\textbf{[#1]}\marginpar{\Large{\textbf{??}}}}

\numberwithin{equation}{section}

\newtheorem{theorem}{Theorem}[section]

\newtheorem{rem}[theorem]{Remark}

\renewcommand{\tilde}{\widetilde}          % wider `tilde'
\DeclareMathSymbol{\leqslant}{\mathalpha}{AMSa}{"36} % nicer `smaller or equal'
\DeclareMathSymbol{\geqslant}{\mathalpha}{AMSa}{"3E} % nicer `larger or equal'
\DeclareMathSymbol{\eset}{\mathalpha}{AMSb}{"3F}     % nicer `emptyset'
\renewcommand{\leq}{\;\leqslant\;}                   % redef. of < or =
\renewcommand{\geq}{\;\geqslant\;}                   % redef. of > or =

\newcommand{\Z}{\mathbb{Z}}

\def\bi{\begin{itemize}}
\def\ei{\end{itemize}}

\def\bnum{\begin{enumerate}}
\def\enum{\end{enumerate}}
\def\<#1{\langle #1 \rangle}

\def\g{\mathbf{g}}

\newtheorem{thm}{Theorem}
\newtheorem{prop}{Proposition}

\newtheorem{cor}{Corollary}
\newtheorem{defn}{Definition}
\newtheorem{lem}{Lemma}
\usepackage{caption}
\usepackage[capitalise]{cleveref}

\Crefname{lem}{Lemma}{Lemmas}

\Crefname{thm}{Theorem}{Theorems}

\Crefname{cor}{Corollary}{Corollaries}
\Crefname{bij}{Bijection}{Bijections}

\Crefname{rem}{Remark}{Remarks}
\crefname{rem}{remark}{remarks}

\title{Simple formulas for constellations and bipartite maps with prescribed degrees}

\author{Baptiste LOUF \footnote{BL is supported by ERC-2016-STG 716083 "CombiTop".}\\ IRIF, Université Paris Diderot - Paris 7 \\
Bâtiment Sophie Germain, 75205 Paris Cedex 13, France \\
blouf@irif.fr}

\begin{document}

%
 %\vspace{1cm}
%\begin{abstract}
\maketitle
\abstract{
We obtain simple quadratic recurrence formulas counting bipartite maps on surfaces with prescribed degrees (in particular, $2k$-angulations), and constellations. These formulas are the fastest known way of computing these numbers. 

Our work is a natural extension of previous works on integrable hierarchies (2-Toda and KP), namely the Pandharipande recursion for Hurwitz numbers (proven by Okounkov and simplified by Dubrovin--Yang--Zagier), as well as formulas for several models of maps (Goulden--Jackson, Carrell--Chapuy, Kazarian--Zograf).  As for those formulas, a bijective interpretation is still to be found. We also include a formula for monotone simple Hurwitz numbers derived in the same fashion.

These formulas also play a key role in subsequent work of the author with T. Budzinski establishing the hyperbolic local limit of random bipartite maps of large genus.}
%\end{abstract}

\vspace{0.3cm}
\footnotesize

\noindent{\bf Keywords:}  maps, Hurwitz numbers, Toda hierarchy, constellations

%\noindent{\bf MSC 2000 subject classifications:  60D05, 81T40,  81T20.}    

\normalsize

\section{Introduction}
A map is  a combinatorial object describing the embedding up to homeomorphism of a multigraph on a compact oriented surface. A bipartite map is a map with black and white vertices, each edge having a black end and a white end. Constellations are generalizations of bipartite maps with more colors (see Section \ref{sec:def_results} for precise definitions).

Map enumeration has been an important research topic for many years now, going back to Tutte \cite{Tutte} with planar maps. He used analytic techniques on generating functions, and later on, Schaeffer enumerated planar maps bijectively \cite{theseSch}, with many generalizations (see for instance \cite{BDG,BF,AP,CMS,Lep}).  The enumeration of maps was extended to other models: for instance, asymptotic formulas were obtained by Bender and Canfield \cite{BC} for maps of higher genus, by Gao \cite{Gao} for maps with prescribed degrees, and Chapuy \cite{Constellations} for constellations. Another way to count maps is to see them as factorizations of permutations and to use algebraic properties of $\mathfrak{S}_n$. In particular, maps fit in the more general context of weighted Hurwitz numbers (see e.g. \cite{ACEH}). Their generating functions satisfy integrable hierarchies of PDEs that arose from mathematical physics, namely the KP and 2-Toda hierarchies (a good introduction can be found in \cite{Solitons}).

The first numbers that were studied from the point of view of integrable hierarchies were Hurwitz numbers, that enumerate ramified coverings of the sphere. Pandharipande conjectured a recurrence formula for those numbers \cite{Pandharipande}, which was proven by Okounkov \cite{Ok} and later simplified by Dubrovin, Yang and Zagier \cite{Zagier}. Later, recurrence formulas for maps were found, starting with Goulden and Jackson for triangulations \cite{GJ}. They were followed by Carrell and Chapuy for general maps \cite{CC}, and Kazarian and Zograf for bipartite maps \cite{KZ}. All these works start from the fact that an underlying generating function is a "tau function" of an integrable hierarchy, and then use ad-hoc techniques to obtain explicit recurrence formulas. The generality of this second step is not well understood. 
%In particular, although the methods mentionned in the previous page enable to control face degrees, this is not the case of the efficient recurrences of \cite{Zagier,GJ,CC,KZ}.  This raises the question of obtaining similar formulas for other models of maps.
The approach developed in \cite{GJ,CC,KZ} does not generalize to constellations, and neither to controlling face degrees (except for the particular minimal case of triangulations \cite{GJ}). On the other hand, in \cite{Ok,Zagier}, formulas are derived only for Hurwitz numbers unramified at $0$ and $\infty$ (which corresponds to maps without control over the degrees of the faces and/or vertices).

\textbf{Contributions of this article:}
% We show that similar techniques as in \cite{Ok,Zagier} apply more generally in the context of maps, 
We manage to combine these two approaches in the context of maps, and we derive recurrence formulas for bipartite maps with prescribed degrees, allowing us in particular to derive a formula for bipartite $2k$-angulations. We also find recurrence formulas for constellations.

These formulas are, up to our knowledge, the simplest and fastest way to calculate those numbers (in all models, it takes $O(n^2g^3)$ arithmetic operations to calculate the coefficient for $n$ edges and genus $g$, see Remark \ref{rem_fastest}).

In addition to the computational aspect, such recurrence formulas are the only tool we know of in the study of asymptotic properties of large genus maps: the Goulden--Jackson formula played a key role in the recent proof \cite{BudzLouf} of the Benjamini--Curien conjecture \cite{PSHT} of the convergence of random high genus triangulations towards a random hyperbolic map. Similarly, the results of this paper are necessary in the study of random high genus bipartite maps in an article in preparation by T. Budzinski and the author.

\textbf{Structure of the paper:}
In Section \ref{sec:def_results}, we will give precise definitions and state our main results. The rest of the paper presents the main steps of the proof.
The first part of the proof is common to all models: we introduce the "tau function" $\tau$, a certain generating function for constellations. This function, along with some auxilliary functions $\tau_n$, classically satisfies a set of differential equations called the $2$-Toda hierarchy. Our first contribution, inspired by \cite{Ok}, is to link $\tau$ to the $\tau_n$ and derive an equation involving $\tau$ only (Proposition \ref{eq_master}). This will be presented in Section \ref{sec:weighted_hurwitz}.
From this equation, specialized to the model we wish for (bipartite maps or constellations), we perform a few combinatorial operations (that are specific to the model, similarly as in \cite{CC, GJ, KZ}) to obtain our formulas. We will present this in details for bipartite maps in Section \ref{sec:rec_formulas}, and we briefly mention the case of constellations. In Section \ref{sec:additional}, we will present additional models, especially one-faced constellations, and in Section \ref{sec:monotone} we will derive a similar formula for (simple, unramified) monotone Hurwitz numbers.
\section{Definitions and main results}\label{sec:def_results}

\begin{defn}\label{def_map_const}
A map $M$ is the data of a connected multigraph (multiple edges and loops are allowed) $G$ (called the underlying graph) embedded in a compact oriented surface $S$, such that $S\setminus G$ is homeomorphic to a collection of disks (this implies in particular that $S$ is connected). The connected components of $S\setminus G$ are called the \textit{faces}. %Equivalently, $M$ is the data of $G$ and a rotation system which describes the cyclic order of the half-edges around each vertex.
The \textit{genus} $g$ of $M$ is the genus of $S$ (the number of "handles" in $S$). $M$ is defined up to orientation-preserving homeomorphism.
A \textit{bipartite map} is a map with two types of vertices (black or white), such that each edge connects two vertices of different colors. A bipartite map is said to be \textit{rooted} if a particular edge is distinguished. 

An \textit{$m$-constellation} is a particular kind of map with two kinds of vertices: colored vertices, carrying a "color" between $1$ and $m$, and star vertices. Each edge connects a star vertex to a colored vertex. A star vertex has degree $m$, and its neighbors have color $1$,$2$,…,$m$ in the clockwise cyclic order. A constellation is said to be rooted if a particular star vertex is distinguished. A constellation with $n$ star vertices is said to be \textit{labeled} if each star vertex carries a different label between $1$ and $n$. Since rooting kills all possible automorphisms, there is a $(n-1)!$-to-$1$ correspondence between labeled and rooted constellations with $n$ star vertices. From now on, we will only consider rooted objects unless stated otherwise. 
\end{defn}

Some basic, well-known, properties of maps and constellations will be useful later.
\begin{prop}\label{prop_constellation_bipartite}
Labeled (non-necessarily connected) $m$-constellations with $n$ star vertices are in bijection with $(m+1)$-uples $(\sigma_1,\sigma_2,…,\sigma_m,\phi)$ of permutations of $\mathfrak{S}_n$ such that $\sigma_1\cdot\sigma_2\cdot …\cdot \sigma_m= \phi$. The permutation $\sigma_i$ represents the vertices of color $i$: each vertex is a cycle of $\sigma_i$, and the elements of the cycle represent the neighboring star vertices, in that cyclic order. The permutation $\phi$ encodes the faces, see Figure \ref{constellation} for an example. Bipartite maps are in bijection with $2$-constellations, since each star vertex and its two adjacent edges can be merged into a single edge connecting a black and a white vertex.
\end{prop}

\begin{figure}[!h]

\centering
\includegraphics[scale=1]{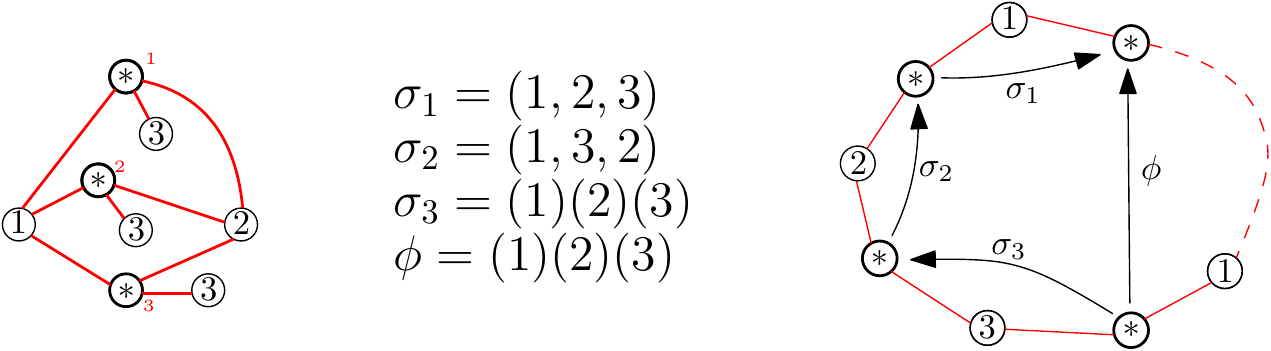}
\caption{Left: a (labeled) $3$-constellation (of genus $0$) and the corresponding permutations, right: the permutation $\phi$, whose cycles describe the faces.}
\label{constellation}
\end{figure}

Our main results are the following theorems:

\begin{thm}\label{thm_biparti_degre}
The number $B_g(\mathbf{f})$ of bipartite maps of genus $g$ with $f_i$ faces of degree $2i$ (for $\mathbf{f}=(f_1,f_2,…)$) satisfies:
\begin{align}\label{rec_biparti_genre}
\binom{n+1}{2}B_g(\mathbf{f})&=\sum_{\substack{\mathbf{s}+\mathbf{t}=\mathbf{f}\\g_1+g_2+g^*=g}}(1+n_1)\binom{v_2}{2g^*+2}B_{g_1}(\mathbf{s})B_{g_2}(\mathbf{t})+\sum_{g^*\geq 0}\binom{v+2g^*}{2g^*+2}B_{g-g*}(\mathbf{f})
\end{align}
where $n=\sum_i if_i$, $n_1=\sum_i is_i$, $v=2-2g+n-\sum_i f_i$, $v_2= 2-2g_2+n_2-\sum_i t_i$ and $n_2=\sum_i it_i$ (the $n$'s  count edges, the $v$'s count vertices, in accordance with the Euler formula), with the convention that $B_g(\mathbf{0})=0$.
\end{thm}
\begin{thm}\label{thm_constellations}
The numbers $C^{(m)}_{g,n}$ of $m$-constellations of genus $g$ with $n$ star vertices satisfy the following recurrence formula:
\[\binom{n}{2}C^{(m)}_{g,n}=\sum_{\substack{n_1+n_2=n\\n_1,n_2\geq 1\\g=g_1+g_2+g^*}} n_1\binom{(m-1)n_2+2-2g_2}{2g^*+2}C^{(m)}_{g_1,n_1}C^{(m)}_{g_2,n_2}.\]
\end{thm}

Theorem \ref{thm_biparti_degre} has an immediate corollary, i.e. a recurrence formula for bipartite $2k$-angulations:

\begin{cor}\label{thm_angulations}
The number $A^{(k)}_{g,n}$ of bipartite $2k$-angulations of genus $g$ with $n$ faces satisfies the following recurrence formula:
\begin{align*}
\binom{kn+1}{2}A^{(k)}_{g,n}=&\sum_{\substack{n_1+n_2=n\\n_1,n_2\geq 1\\g_1+g_2+g^*=g}} (kn_1+1) \binom{(k-1)n_2+2-2g_2}{2g^*+2}A^{(k)}_{g_1,n_1}A^{(k)}_{g_2,n_2}\\&+\sum_{g^*\geq 0}\binom{(k-1)n+2-2(g-g^*)}{2g^*+2}A^{(k)}_{g-g*,n}.
\end{align*}
\end{cor}

\begin{rem}\label{rem_fastest}
Theorem \ref{thm_biparti_degre} allows to compute the number of maps with prescribed degrees way faster than the usual Tutte-Lehman-Walsh approach \cite{LW,BC,Gao} or the topological recursion (see e.g. \cite{Eynard}), especially for large genus (because these methods require counting maps with up to $g$ boundaries to enumerate maps of genus $g$). It can also be specialized to maps with bounded face degrees (contrarily to the Tutte equation). Note that, in order to compute the coefficients recursively, a term from the RHS has to be moved to the LHS, and we need the initial condition $B_0((\mathbb{1}_{i=n}))=Cat(n)$.

We observe that Theorem \ref{thm_constellations} applies to bipartite maps (for $m=2$). However, we have no analogue of Theorem \ref{thm_biparti_degre} (with prescribed face degrees) for $m$-constellations with $m\geq 3$. We give a brief explanation of that fact in Remark \ref{rem_no_face_tracking}.

%We can also derive a similar-looking formula for monotone simple Hurwitz numbers, see \eqref{eq_monotone}.
\end{rem}

\begin{rem}
The coefficients in our recurrence formulas have a combinatorial flavor. It is a natural question to ask for a bijective proof of these formulas. However, the bijective interpretation of formulas  derived from the KP/2-Toda hierarchies is still a widely open question, as bijections have only been found for certain formulas, in the particular cases of one-faced \cite{CFF} and planar maps \cite{Louf}. Note that there is an asymmmetry in the factors in the quadratic sums, contrarily to the formulas in \cite{GJ,CC}, but similarly to \cite{KZ}.
\end{rem}

\section{Constellations and the Toda hierarchy}\label{sec:weighted_hurwitz}

\subsection{The semi-infinite wedge space}

We give some definitions, mostly following the notations of the appendix in \cite{Ok2}: 
\begin{defn}
A \textit{Maya diagram} is a decoration of $\Z+\frac{1}{2}$ with a particle or an antiparticle at each position, such that for some $n_1,n_2$ there are only particles at positions $<n_1$ and only antiparticles at positions $>n_2$. The semi infinite wedge space $\Lambda^{\frac{\infty}{2}}$ is the vector space generated by the Maya diagrams. It is equipped with an inner product by making the Maya diagrams orthogonal to each other and of norm $1$.

For any $k\in \Z+\frac{1}{2}$, we define the \textit{fermion operators} $\psi_k$ and $\psi_k^*$. For each Maya diagram $\mathbf{m}$, we set:
\begin{align*}
\psi_k \mathbf{m}= \begin{cases}
0 & \text{if } \mathbf{m} \text{ has a particle in position } k\\
(-1)^{n_k}\tilde{\mathbf{m}} & \text{otherwise}\\
\end{cases}
\end{align*}
% and 
\begin{align*}
 \psi_k^* \mathbf{m}= \begin{cases}
0 & \text{if } \mathbf{m} \text{ has an antiparticle in position } k\\
(-1)^{n_k}\overline{\mathbf{m}} & \text{otherwise}\\
\end{cases}
\end{align*}
where $n_k$ is the number of particles of $\mathbf{m}$ is positions $>k$ (it is finite by definition of a Maya diagram). Also, $\tilde{\mathbf{m}}$ is the same as $\mathbf{m}$ except there is a particle in position $k$, and $\overline{\mathbf{m}}$ is the same as $\mathbf{m}$ except there is an antiparticle in position $k$. Note that $\psi_k$ and $\psi_k^*$ are adjoint operators.

We can now define the \textit{boson operators}: for all $n\in\Z^*$, let \[\alpha_n=\sum_{\substack{k\in \Z+\frac{1}{2}}} \psi_{k-n}\psi_k^*.\]

Finally, the two \textit{vertex operators} are  \[\Gamma_\pm(\mathbf{p})=\exp(\sum_{n=1}^\infty  \frac{p_n}{n}\alpha_{\pm n}).\]
\end{defn}

We will now define diagonal operators over $\Lambda^{\frac{\infty}{2}}$ and relate Maya diagrams to partitions.

\begin{defn}
We define the \textit{normally ordered products} 
\[:\psi_k\psi_k^*:=\begin{cases}
\psi_k\psi_k^* &\text{if } k>0\\
-\psi_k^*\psi_k &\text{if } k<0.\\

\end{cases}\]

Note that, for a Maya diagram $\mathbf{m}$
\[:\psi_k\psi_k^*:\mathbf{m}=\begin{cases}
\mathbf{m} &\text{if } k>0 \text{ and } \mathbf{m} \text{ has a particle in position } k\\
-\mathbf{m} &\text{if } k<0 \text{ and } \mathbf{m} \text{ has an antiparticle in position } k\\
0 &\text{otherwise.}\\

\end{cases}\]
The \emph{charge operator} is:
\[C=\sum_{\substack{k\in \Z+\frac{1}{2}}} :\psi_k\psi_k^*: .\]
The eigenvectors of $C$ are the Maya diagrams. The eigenvalue of a Maya diagram $\mathbf{m}$ is the number of particles in positive position minus the number of antiparticles in negative position. We call this number the \textit{charge} of $\mathbf{m}$.
We introduce the translation operator $R$: for any $\mathbf{m}$, $R\mathbf{m}$ has a particle in position $k+1$ if and only if $\mathbf{m}$ has a particle in position $k$. Note that if the charge of $\mathbf{m}$ is $c$, the charge of $R\mathbf{m}$ is $c+1$, and that the adjoint of $R$ is $R^{-1}$.

There is a bijection between Maya diagrams of charge $0$ and partitions, as depicted in Figure \ref{maya_diagram} (in position $k$, a down-step corresponds to a particle, an up-step corresponds to an antiparticle). Thus, any Maya diagram $\mathbf{m}$ can be encoded by its charge $c$ and a partition $\lambda$ (that corresponds to the Maya diagram $R^{-c} \mathbf{m}$). 
\begin{figure}[!h]
\centering
\includegraphics[scale=0.3]{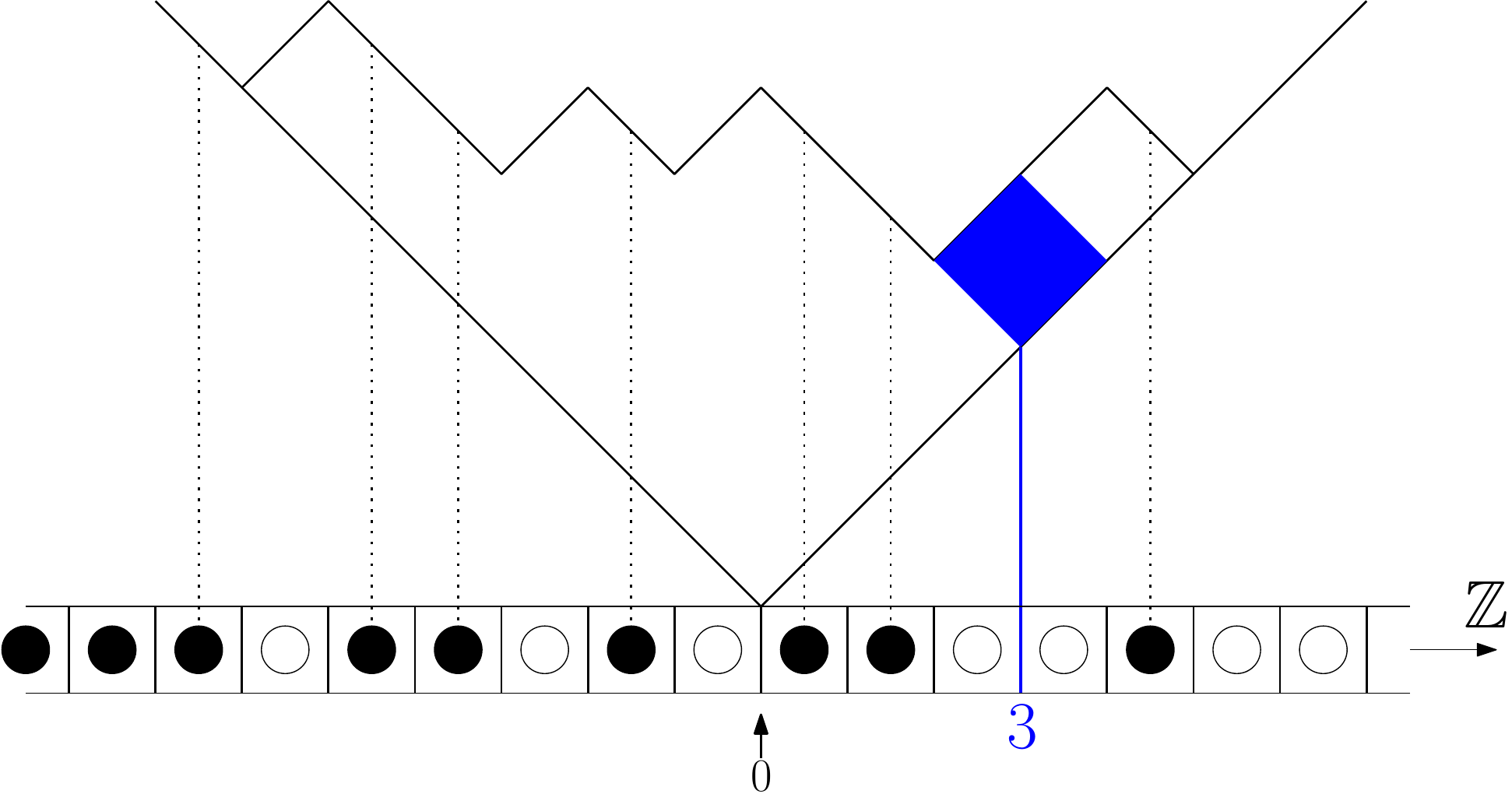}
\caption{A Maya diagram of charge $0$ and its corresponding partition (above it, presented as a rotated Young diagram). Particles are in black. In blue, a box and its content (the abscissa of the projection of the center of the box on $\Z$).}
\label{maya_diagram}
\end{figure}

We will use the braket notation, and denote the Maya diagram corresponding to the empty partition by $\ket{\emptyset}$, and set $\ket{\emptyset_n}=R^n\ket{\emptyset}$. We will also set $\ket{\lambda}$ to be the Maya diagram of charge $0$ corresponding to the partition $\lambda$. 

Finally, we define the \emph{energy operator}
\[H=\sum_{\substack{k\in \Z+\frac{1}{2}}} k:\psi_k\psi_k^*:.\]
In particular,  $H\ket{\lambda}=|\lambda| \ket{\lambda}$, where $|\lambda |$ is the number of boxes in $\lambda $. 

\end{defn}

\subsection{Generating functions as tau functions}

\begin{defn}\label{def_W}
Fix integers $r$, $n$ and $g$, fix $\lambda,\mu$ two partitions of $n$. Let $W_{n}^{\lambda,\mu}(l_1,l_2,…,l_r)$ be the number of $(r+2)$-uples of permutations $(\sigma_1,\sigma_2,…,\sigma_{r},\sigma_\lambda,\sigma_\mu)$ of $\mathfrak{S}_n$ such that $\sigma_1\cdot \sigma_2\cdot …\cdot \sigma_{r}=\sigma_\lambda\sigma_\mu$ and $\sigma_i$ has $l_i$ cycles, and $\sigma_\lambda,\sigma_\mu$ have respective cycle types $\lambda$ and $\mu$.
The $W_{n}^{\lambda,\mu}$ enumerate (labeled, non-necessarily connected) constellations, in accordance with Proposition \ref{prop_constellation_bipartite}. Let $\tau$ be the associated generating function (that implicitly depends on $r$):
\[ \tau(z,\mathbf{p},\mathbf{q},(u_j)) =\sum_{\substack{n\geq 0\\|\mu |=|\lambda |=n\\ l_i\geq 1\text{ }  \forall i} }\frac{z^n}{n!} \prod_{i=1}^r u_i^{n-l_i} p_\lambda q_\mu W_{n}^{\lambda,\mu}(l_1,l_2,…,l_r) \]
\end{defn}

\begin{rem} 
Depending on the specialization that will be applied, these $(r+2)$-uples will either represent $r$- or $(r-1)$-constellations.
\end{rem}
It is a classical result (under different forms and variants, see for instance \cite{GJ,Ok}) that the function $\tau$ can be expressed in terms of elements and operators of $\Lambda^{\frac{\infty}{2}}$:
\begin{lem}[Classical]\label{lem_operator_form}
\begin{equation}\label{eqn_tau_operators}
\tau(z,\mathbf{p},\mathbf{q},(u_j))=\mel{\emptyset}{\Gamma_+(\mathbf{p})z^H\Lambda \Gamma_-(\mathbf{q})}{\emptyset}
\end{equation}
with 
\[F(u)=\sum_{\substack{k>0}}\sum_{\substack{i=0}}^{k-1/2}\log(1+ui)\psi_k\psi_k^*+\sum_{\substack{k<0}}\sum_{\substack{i=0}}^{-k-1/2}\log(1-ui)\psi_k^*\psi_k\]
and  $\Lambda=\prod_{j=1}^r\exp(F(u_j))$.

\end{lem}
%\louf{réécrire cette preuve}
\begin{proof}
First, we have 
\[F(u)\ket{\nu}=\left(\sum_{\Box\in\nu} \log(1+uc(\Box)) \right)\ket{\nu}\]
where the $c(\Box)$ are the contents of the partition $\nu$ (see Figure \ref{maya_diagram}). It can be shown using the Jacobi-Trudi rule (see e.g. \cite{Ok2}) that 

\[\Gamma _ { - } ( \mathbf{q}) \ket { \emptyset } = \sum _ { \nu } s _ { \nu } ( \mathbf{q}) \ket { \nu }\quad \text{and}\quad \bra { \emptyset }\Gamma _ {+} ( \mathbf{p})  = \sum _ { \nu } s _ { \nu } ( \mathbf{p}) \bra { \nu }\]
where the sum spans over all partitions.

Thus the RHS of \eqref{eqn_tau_operators} can be rewritten as:
\[\sum_{\substack{n>0\\|\nu|=n}} z^n \prod_{j=1}^r  \prod_{\Box\in\nu} (1+u_jc(\Box))s _ { \nu } ( \mathbf{p}) s _ { \nu } ( \mathbf{q}).\]

%Here we recover the "content product form" of the generating function of constellations, as in \cite{GJ}, which concludes the proof.
This expression (the "content product form") is equal to $\tau(z,\mathbf{p},\mathbf{q},(u_j))$ (see e.g. \cite{GJ}, Theorem 3.1).

\end{proof}
We introduce the auxiliary functions $\tau_n$, for $n\in\Z$:
 \[\tau_n =\mel{\emptyset_n}{\Gamma_+(\mathbf{p})z^H\Lambda \Gamma_-(\mathbf{q})}{\emptyset_n}.\]
We have $\tau=\tau_0$. The previous lemma, along with classical considerations (see for instance Section 2.6 in \cite{Ok}), imply that the $\tau_n$ satisfy an infinite set of equations, the $2$-Toda hierarchy. In particular, the following equation holds:

\begin{equation}\label{eq_tau_pm_1}
\frac { \partial ^ { 2 } } { \partial p _ { 1 } \partial q _ { 1 } } \log \tau_0  = \frac { \tau _ {  1 } \tau _ { - 1 } } { \tau_0 ^ { 2 } }.
\end{equation}

So far, the content presented was classical. Our first main contribution is to transform the previous equation into an equation implying $H=\log\tau$ only.

\subsection{The master equation}
In this section we derive the following general equation:

\begin{prop}\label{prop_master}
The general generating function of connected constellations $H=\log\tau$ satisfies:
\begin{equation}\label{eq_master}
DH_{1,1}-H_{1,1}=H_{1,1}\left(DH\left(z\cdot\prod_{j=1}^r (1+ u_j),(\frac{u_j}{1+ u_j})\right)+DH\left(z\cdot\prod_{j=1}^r (1- u_j),(\frac{u_j}{1- u_j})\right)-2DH\right)
\end{equation}
with $H_{1,1}=\frac { \partial ^ { 2 } } { \partial p _ { 1 } \partial q _ { 1 } } H$ and $D=z\frac{\partial}{\partial z}$.
\end{prop}

\begin{rem}

In the formula above, we omitted some of the arguments of H. For instance, $H:=H(z,\mathbf{p},\mathbf{q},(u_j))$, and $H\left(z\cdot\prod_{j=1}^r (1+ u_j),(\frac{u_j}{1+ u_j})\right):=H\left(z\cdot\prod_{j=1}^r (1+ u_j),\mathbf{p},\mathbf{q},(\frac{u_j}{1+ u_j})\right)$.
%We only included the arguments of $H$ when they differ from $(z,\mathbf{p},\mathbf{q},(u_j))$.

%Note that there is no ambiguity in writing something like $DH\left(z\cdot\prod_{j=1}^r (1+ u_j),(\frac{u_j}{1+ u_j})\right)$, since $(Df)(az)=D(z\rightarrow f(az))$.
\end{rem}

This formula will be the starting point for all the particular cases we will consider in the next section: for each model, we will apply a particular specialization of the variables, then interpret combinatorially the operator $\frac{\partial}{\partial p_1}\frac{\partial}{\partial q_1}$ (depending on the model), and finally the extraction of coefficients will give us the relevant formulas.

We first need to relate the auxilliary functions $\tau_1$ and $\tau_{-1}$ to the generating function $\tau$. 

\begin{lem}\label{lem_tau_aux}
\[\tau_{\pm 1}(z,\mathbf{p},\mathbf{q},(u_j))=z^{1/2}\tau\left(z\cdot\prod_{j=1}^r (1\pm u_j),\mathbf{p},\mathbf{q},(\frac{u_j}{1\pm u_j})\right)\]
\end{lem}

\begin{proof}
We will describe how $H$, $C$ and $F$ behave under the action of the shift operator, then using the operator form \eqref{eqn_tau_operators} of $\tau$ we will derive the result.

It is easily verified that the opertors $R$, $C$ and $H$ commute with $\Gamma_+(\mathbf{p})$ and $\Gamma_-(\mathbf{q})$. We also have $R^{-n}HR^n=H+nC+\frac{n^2}{2}$.

By a careful change of indices,
\begin{align*}
 RFR^{-1}&=\sum_{\substack{k>0}}\sum_{\substack{i=1}}^{k-1/2}\log(1+ui)\psi_{k+1}\psi_{k+1}^*+\sum_{\substack{k<0}}\sum_{\substack{i=1}}^{-k-1/2}\log(1-ui)\psi_{k+1}^*\psi_{k+1}\\
 %&=\sum_{\substack{k>1}}\sum_{\substack{i=0}}^{k-3/2}\log(1-u+ui)\psi_k\psi_k^*+\sum_{\substack{k<1}}\sum_{\substack{i=0}}^{-k+1/2}\log(1--u-ui)\psi_k^*\psi_k\\
% &=\sum_{\substack{k>0}}\sum_{\substack{i=1}}^{k-3/2}\log(1-u+ui)\psi_k\psi_k^*+\sum_{\substack{k<0}}\sum_{\substack{i=1}}^{-k+1/2}\log(1--u-ui)\psi_k^*\psi_k\\
&=\sum_{\substack{k>0}}\sum_{\substack{i=0}}^{k-1/2}\log(1-u+ui)\psi_{k}\psi_{k}^*+\sum_{\substack{k<0}}\sum_{\substack{i=0}}^{-k-1/2}\log(1-u-ui)\psi_{k}^*\psi_{k}\\
 &=F(\frac{u}{1-u})+(H-C/2)\log(1-u).
\end{align*}

Since $\ket{\emptyset_{-1}}=R^{-1}\ket{\emptyset}$, we have
\begin{align*}
\tau_{-1}(z,\mathbf{p},\mathbf{q},(u_j))&=\mel{\emptyset_{-1}}{\Gamma_+(\mathbf{p})z^H\Lambda \Gamma_-(\mathbf{q})}{\emptyset_{-1}}\\
&=\mel{\emptyset}{R\Gamma_+(\mathbf{p})z^H\Lambda \Gamma_-(\mathbf{q})R^{-1}}{\emptyset}\\
&=\mel{\emptyset}{\Gamma_+(\mathbf{p})z^{RHR^{-1}}R\Lambda R^{-1} \Gamma_-(\mathbf{q})}{\emptyset}\\
&=\mel{\emptyset}{\Gamma_+(\mathbf{p})z^{H-C+1/2}\prod_{j=1}^r\exp\left(F(\frac{u_j}{1-u_j})+(H-C/2)\log(1-u_i) \right)\Gamma_-(\mathbf{q})}{\emptyset}\\
&=z^{1/2}\mel{\emptyset}{\Gamma_+(\mathbf{p})(z\prod_{j=1}^r (1-u_j))^H\Lambda\left((\frac{u_j}{1-u_j})\right)\Gamma_-(\mathbf{q})}{\emptyset}\\
&=z^{1/2}\tau\left(z\cdot\prod_{j=1}^r (1- u_j),\mathbf{p},\mathbf{q},(\frac{u_j}{1- u_j})\right).
\end{align*}

Similarly, \[\tau_{ 1}(z,\mathbf{p},\mathbf{q},(u_j))=z^{1/2}\tau\left(z\cdot\prod_{j=1}^r (1+ u_j),\mathbf{p},\mathbf{q},(\frac{u_j}{1+ u_j})\right).\]

\end{proof}

\begin{rem}
The idea of expressing $\tau_{\pm 1}$ in terms of $\tau$ by calculating $R^{\mp 1}\Lambda R^{\pm 1}$ is inspired by the calculation performed in \cite{Ok}, Section 2.7.
\end{rem}

We can now prove Proposition \ref{prop_master}.

\begin{proof}[Proof of Proposition \ref{prop_master}]

Using Lemma \ref{lem_tau_aux}, we can interpret \eqref{eq_tau_pm_1} as an equation implying $\tau$ only:

\[\frac { \partial ^ { 2 } } { \partial p _ { 1 } \partial q _ { 1 } } \log \tau  = z\frac { \tau\left(z\cdot\prod_{j=1}^r (1+ u_j),\mathbf{p},\mathbf{q},(\frac{u_j}{1+ u_j})\right)\tau\left(z\cdot\prod_{j=1}^r (1- u_j),\mathbf{p},\mathbf{q},(\frac{u_j}{1- u_j})\right) } { \tau  ^ { 2 } }.\]

Substituting $H=\log\tau$ in the above equation, one obtains:
\begin{equation}\label{eq_H_exp}
H_{1,1}=z\exp\left(H(z\cdot\prod_{j=1}^r (1+ u_j),(\frac{u_j}{1+ u_j}))+H(z\cdot\prod_{j=1}^r (1- u_j),(\frac{u_j}{1- u_j}))-2H\right).
\end{equation}

Finally, we get \eqref{eq_master} by applying the operator $D-1$ to both sides of \eqref{eq_H_exp} and getting rid of the exponential part by using \eqref{eq_H_exp} another time.

\end{proof}
\section{Proof of the main formulas}\label{sec:rec_formulas}
In the following subsections, we will specialize some of the variables to fit the cases we care about. To avoid tedious notations, and as there is no risk of ambiguity, the specialization of the function $H$ will still be called $H$.
\subsection{Bipartite maps}

In this section, we want to count bipartite maps while controlling the degrees of the faces. Thus, we will consider the case $r=2$, and specialize $H$ by setting $u_1=u_2=u$ and $q_i=\mathbb{1}_{i=1}$.

Let $B_g(\mathbf{f})$ be the number of (rooted) bipartite maps of genus $g$ with $f_i$ faces of degree $2i$, and $B(z,\mathbf{p},u)$ be the ordinary generating function of connected rooted bipartite maps, defined as
\[B=\sum_{\substack{g,\mathbf{f}}} z^n u^{2n-v} \prod_{i\geq 1}p_i^{f_i} B_g(\mathbf{f}).\]
with $n=\sum_i if_i$ and $v-n+\sum_i f_i = 2-2g$ (Euler formula).

Equation \eqref{eq_master} can be rewritten in terms of $B$ only:
\begin{lem}
\begin{equation}\label{eq_biparti_GF}
(D+1)DB=(u^{-2}+(D+1)B)\left( B(z(1+u)^2,\mathbf{p},\frac{u}{1+u})+B(z(1-u)^2,\mathbf{p},\frac{u}{1-u})-2B \right).
\end{equation}
\end{lem}

\begin{proof}
In this section, $H$ is the (exponential) generating function of labeled bipartite maps, and as mentioned in Definition \ref{def_map_const}, there is a $(n-1)!-to-1$ correspondence between labeled and rooted bipartite maps. Thus
\[B=DH\]

We will now express $H_{1,1}$ in terms of $B$. The specialization $q_i=\mathbb{1}_{i=1}$ implies that only the terms $z^nq_1^n$ form the original function survived, and thus in this case \[\frac{\partial}{\partial q_1}H=DH.\] Finally, applying $\frac{\partial}{\partial p_1}$ corresponds to marking a digon. 
A marked digon can be contracted into a marked edge (see Figure \ref{fig_digon_contraction}) except when the bipartite map is just one edge, thus $\frac{\partial}{\partial p_1}H=z+u^2zDH=z+u^2zB$ (the $u^2z$ factor comes from the fact that we lose an edge when we contract the digon, and the $z$ term is the case where we cannot contract the digon). 
\end{proof}

We are finally ready to prove Theorem \ref{thm_biparti_degre}.

\begin{proof}[Proof of Theorem \ref{thm_biparti_degre}]
We look at the factor $B(z(1+u)^2,\mathbf{p},\frac{u}{1+u})+B(z(1-u)^2,\mathbf{p},\frac{u}{1-u})-2B$ in \eqref{eq_biparti_GF}.
The coefficient of $z^n \prod_{i\geq 1} p_i^{f_i}$ in it is:
\begin{align*}
\sum_{v>0} B_g(\mathbf{f})u^{2n-v}\left( (1+u)^{v}+(1-u)^v-2\right)=\sum_{v>0}B_g(\mathbf{f})u^{2n-v}\ \left( 2 \sum_{0<k\leq\frac{v}{2}} u^{2k}\binom{v}{2k}\right).
\end{align*}

In the sum above, we have, by Euler's formula,  $g= \frac{n-\sum f_i -v+2}{2}$ (with the convention that $B_g(\mathbf{f})=0$ if $g$ is not an integer).
Extracting the coefficient of $z^n u^{2n-v}\prod_{i\geq 1} p_i^{f_i}$  in \eqref{eq_biparti_GF}, one gets the result.
\end{proof}

\begin{figure}
\center
\includegraphics[scale=0.8]{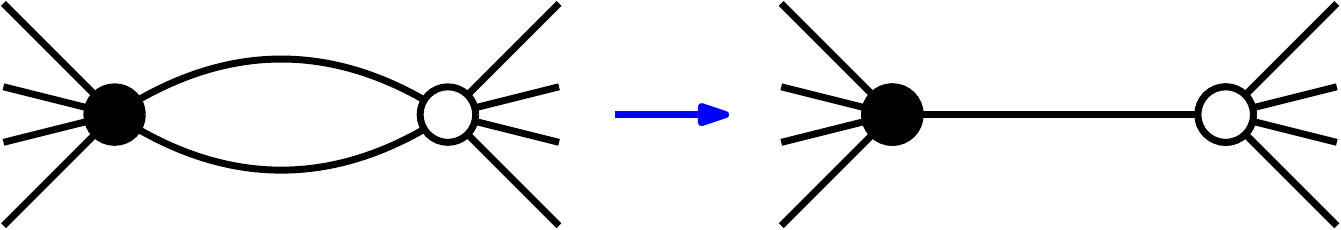}
\caption{Contracting a digon.}\label{fig_digon_contraction}
\end{figure}
\subsection{Constellations}

In this section, we will count constellations without controlling the degrees of the faces. For that, we will specialize $H$ by taking $r=m+1$, $p_i=q_i=\mathbb{1}_{i=1}$, and $u_i=u$ for all $i$. The variable $u$ counts the number of colored vertices plus the number of faces, or equivalently, by Euler's formula, the genus.

\begin{proof}[Proof of Theorem \ref{thm_constellations}]
After the specialization, $H_{1,1}$ becomes $D^2 H$ by the same argument as in the proof of Theorem \ref{thm_biparti_degre}. If we take $C$ to be the (ordinary) generating function of connected constellations, i.e.
\[C=\sum_{g,n}z^n u^{2n+2g-2}C^{(m)}_{g,n}, \]
we have, as before, $C=DH$. Equation \eqref{eq_master} becomes
\begin{equation}\label{eq_const_GF}
(D^2-D)C=DC\left(C\left(z(1+u)^{m+1},\frac{u}{1+ u}\right)+C\left(z(1- u)^{m+1},\frac{u}{1- u})\right)-2C\right)
\end{equation}

To finish the proof, we proceed exactly as in the proof of Theorem \ref{thm_biparti_degre}: first calculate the coefficient of $z^n$ in 
\[C\left(z(1+u)^{m+1},\frac{u}{1+ u}\right)+C\left(z(1- u)^{m+1},\frac{u}{1- u})\right)-2C,\]
then just extract the coefficient of $z^n u^{2n+2g-2}$ in \eqref{eq_const_GF} (the suitable exponent of $u$ is derived by the Euler formula).
\end{proof}

\begin{rem}\label{rem_no_face_tracking}
This time, we cannot track the degrees of the faces, as in general the combinatorial operation of contracting an $m$-gon might disconnect the map, and the formula gets messy. However, if we restrict to only one face we can perform this operation to recover a nice formula (see Section \ref{sec:additional}).
\end{rem}
\section{Additional results}\label{sec:additional}

\subsection{One-faced constellations}

In this section, we will derive a recurrence formula for constellations with one face. In the case of bipartite maps, the formula is just a particular case of \eqref{rec_biparti_genre}, but for $m\geq 3$, it cannot be derived from Theorem \ref{thm_constellations} directly. One-faced constellations were first enumerated in \cite{PS}: an exact formula given the degree distribution of each colored vertex is provided. While the following formula does not give control over the degrees of the vertices, it is much quicker to calculate the "global" (i.e. controlling only the genus and the number of vertices) number of one-faced constellations for $m \geq 3$ (for $m=2$, i.e. bipartite maps, a nice formula for one-faced bipartite maps can be found in \cite{Adrianov}).

\begin{thm}\label{thm_unicellular}
Let $U_m(g,n)$ be the number of one-faced $m$-constellations of genus $g$ with $n$ star vertices. Also, let $U^{(k)}_m(g,n)$ be the number of one-faced $m$-constellations of genus $g$ with $n$ star vertices and $k$ distinguished (pairwise distinct) colored vertices, i.e. $U^{(k)}_m(g,n)=\binom{(m-1)n+1-2g}{k}U_m(g,n)$. We have the following recurrence formula:
\begin{equation}\label{eq_unicellular}
\frac{n(n+1)^{m-1}}{2}U_m(g,n)=\sum_{g^*=0}^g U^{(2+2g^*)}_m(g-g^*,n)
\end{equation}
\end{thm}

\begin{rem}
This formula reminds of the formula for one-faced maps proven bijectively by Chapuy in \cite{trisections}. Indeed, it allows to calculate the number of one-faced maps of genus $g$ in terms of number of maps of lower genus with the same number of edges and some distinguished vertices. The difference, although, is that in Chapuy's formula there are an odd number of distinguished vertices, whereas in \eqref{eq_unicellular} there are an even number of distinguished vertices.

% Nevertheless, the fact that they both appear in a similar context (as an intermediate step to prove a formula issued from the KP hierarchy on the one hand, as a formula issued from the Toda hierarchy on the other hand) suggests that there might still be some kind of connection. This will be left as an open question.

Nevertheless, there might be a connection as those formulas arise in the same algebraic context. Our formula is obtained via the 2-Toda hierarchy, whereas Chapuy's is an intermediate step to prove the Harer-Zagier recurrence formula (see \cite{CFF}), which is itself a special case of a formula obtained via the KP hierarchy: the Carrell-Chapuy recurrence formula \cite{CC}.
\end{rem}

To prove \eqref{eq_unicellular}, we will take $r=m$ and apply the following specialization to \eqref{eq_master}: fix an integer $n$, and set $q_i=\mathbb{1}_{i=1}$, $u_i=u$ for all $i$, as well as $z=1$. Set also $p_i=0$ for all $i\neq n$, and extract the coefficient of $p_n^1$. $H$ is now simply a polynomial in $u$. It counts labeled one-faced constellations. Let $U$ be the associated polynomial for rooted objects, the classical correspondence between labeled and rooted objects yields $U=DH$. As before, there is a "marked $m$-gon", and we need to interpret this combinatorially:

\begin{lem}
After the specialization, the LHS of \eqref{eq_master} becomes $n(n+1)^{m-1}U$
\end{lem}

\begin{proof}
The only terms of $H_{1,1}$ in \eqref{eq_master} that survive the specialization are the coefficients of $z^{n+1}p_np_1q_1^{n+1}$. Thus we have $DH_{1,1}=(n+1)H_{1,1}$ and $H_{1,1}=(n+1)\frac{\partial}{\partial p_1}H$. The LHS of \eqref{eq_master} is therefore equal (after specialization) to $n(n+1)\frac{\partial}{\partial p_1}H$. It remains to show that $\frac{\partial}{\partial p_1}H=(n+1)^{m-2}\cdot U$. 

Applying $\frac{\partial}{\partial p_1}$ corresponds to marking an $m$-gon. As in the proof of Theorem \ref{thm_biparti_degre}, it kills all symmetries, thus there is a $(n+1)!$-to-$1$ correspondence between labeled constellations and constellations with a marked $m$-gon. Therefore, $\frac{\partial}{\partial p_1}H$ is the ordinary generating function of connected unlabeled $m$-constellations with one face of degree $mn$ and one face of degree $m$.

We will work with permutations to make things easier. Connected unlabeled $m$-constellations with one face of degree $mn$ and one face of degree $m$ are in bijection with $(m+1)$-uples of permutations $\phi, \sigma_1, … \sigma_m$ of $\mathfrak{S}_{n+1}$ satisfying the following constraints:
\begin{itemize}
\item  $\phi=\prod_{i=1}^m \sigma_i$
\item In cycle products, $\phi$ is written $(1,2,…,n)(n+1)$.
\item The image of $1$ by $\sigma_1$ is $n+1$.
\end{itemize}

We can describe the operation of  "contracting an $m$-gon" on the permutations. To $\phi, \sigma_1, … \sigma_m$ we will associate a $(m+1)$-uple $\phi', \sigma_1', … \sigma_m'$ of permutations of $\mathfrak{S}_{n}$:
\begin{itemize}
\item To $\phi$, we associate $\phi'=(1,2,…,n)$
\item For $1\leq i<m$, to $\sigma_i$ we associate the permutation $\sigma_i'$ where in the cycle product we just deleted the element $n+1$ (see Figure \ref{contract_mgon})
\item To $\sigma_m$ we associate $\sigma_m'=\phi'\left(\prod_{i=1}^{m-1}\sigma_i'\right)^{-1}$
\end{itemize}
This exactly describes a rooted $m$-constellation with one face of degree $mn$. To go back, one needs to remember, for $1<i<m$, what was the preimage of $n+1$ in $\sigma_i$ (including possibly $n+1$ itself). There are $n+1$ possible choices for each $i$, thus after the specialization, $\frac{\partial}{\partial p_1}H=(n+1)^{m-2}\cdot U$.

\begin{figure}
\center
\includegraphics[scale=0.8]{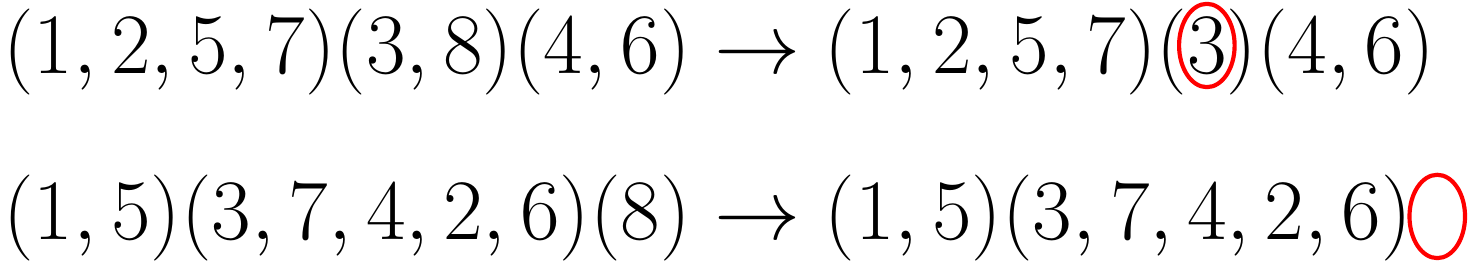}

\caption{Deleting $n+1$, for $n=7$, whether $n+1$ is a fixed point or not}
\label{contract_mgon}
\end{figure}
\end{proof}

A simple calculation in the right-hand side finishes the proof:

\begin{proof}[Proof of Theorem \ref{thm_unicellular}]
In the RHS, we have a product of two terms. Since $H$ has no constant coefficient in the $p_i$'s, after specialization we get the coefficient of $p_n^0$ of $H_{1,1}$ (which is just $u^{m-1}$, corresponding to the constellation with only one star vertex) times the coefficient of $p_n^1$ in 
\[DH\left((1+u )^m,\frac{u}{1+ u}\right)+DH\left((1- u)^m,\frac{u}{1- u}\right)-2DH.\]
Again, since $DH=U$, we can extract the coefficient of $z^nu^{mn-v}$ (where $v=(m-1)n+1-2g$ by Euler's formula), as in the proof of Theorem \ref{thm_biparti_degre}, and obtain the result.
\end{proof}

\subsection{Controlling more parameters}

In each of the previous cases, we specialized a lot of variables to obtain formulas for "global" coefficients. Starting over from \eqref{eq_master} without specializing some of the variables, one is able to obtain (slightly more complicated) formulas for more fine-grained coefficients. As an example, we can calculate the number $C^{(m)}_{g,n,f}$ of $m$-constellations of genus $g$, with $n$ star vertices and $f$ faces:
\begin{equation}\label{eq_const_faces}
\binom{n}{2}C^{(m)}_{g,n,f}=\sum n_1 \binom{f_2}{k}\binom{2g_2-f_2+(m-1)n_2}{2g^*+2-k}C^{(m)}_{g_1,n_1,f_1}C^{(m)}_{g_2,n_2,f_2}
\end{equation}
where the sum is over $n_1+n_2=n$, $n_1,n_2>0$, $g^*\geq 0$, $g_1+g_2+g^*=g$ and  $f_1+f_2-k=f$.

The proof of Theorem \eqref{eq_const_faces} is essentially the same as the proof of Theorem \ref{thm_constellations}, except that we do not specialize $u_i=u$ for all $i$, but only for $i\leq m$. In this case, $u$ counts colored vertices, and $u_{m+1}$ counts faces.

\begin{rem}
Even though the summation is complicated, \eqref{eq_const_faces} allows to compute all the coefficients $C^{(m)}_{g,n,f}$ from the initial condition $C^{(m)}_{g,1,f}=1$ iff $g=0$ and $f=1$, and $0$ otherwise.

However, it does not restrict to a formula for one-faced constellations.
\end{rem}

We can also find formulas for other models, with other specializations. Relevant models include bipartite maps (with prescribed face degrees), one-faced constellations, or (general) constellations, with control over the number of vertices of each color. We can also obtain a formula for triangulations (by specializing $r=1$, $p_i=\mathbb{1}_{i=2}$, $q_i=\mathbb{1}_{i=3}$), but it is more complicated (and less "combinatorial") than the Goulden--Jackson formula \cite{GJ}. The reader is encouraged to play with \eqref{eq_master} to find other nice formulas.

%\subsection{Controlling the colors of the vertices}
%For $\mathbf{v}=(v_1,v_2,…,v_m)$, let $C_{g,n}(\mathbf{v})$ be the number of $m$-constellations of genus $g$ with $n$ star vertices and $v_i$ vertices of color $i$. 
%
%
%\begin{thm}\label{thm_color_tracking}
%The numbers $C_{g,n}(\mathbf{v})$ satisfy the following equation:
%
%\begin{equation}\label{eq_color_tracking}
%\binom{n}{2}C_{g,n}(\mathbf{v})=\sum n_1\prod_{i=1}^m\binom{y_i}{k_i}\binom{2-2g_2+(m-1)n_2-\sum y_i}{2g^*+2-\sum k_i}C_{g_1,n_1}(\mathbf{x})C_{g_2,n_2}(\mathbf{y})
%\end{equation}
%where the sum is over $n_1+n_2=n$, $n_1,n_2>0$, , $g^*\geq 0$, $k_i\leq y_i$, $0\leq 2g^*+2-\sum k_i\leq 2-2g_2+(m-1)n_2-\sum y_i$,  $v_i=x_i+y_i-k_i$ and $g=g_1+g_2+g^*$
%\end{thm}
%
%
%
%The proof of Theorem \ref{thm_color_tracking} is essentially the same as the proof of Theorem \ref{thm_constellations}, except that we do not specialize $u_i=u$ for all $i$. Then, for $i\leq m$, $u_i$ counts vertices of color $i$, and $u_{m+1}$ counts the number of faces, and we use Euler's formula to control the genus.
%
%\begin{rem}
%Although \eqref{eq_color_tracking} has a complicated summation formula, it allows to calculate $C_{g,n}(\mathbf{v})$ when knowing numbers of the form $C_{g',n'}(\mathbf{v'})$ for $g'\leq g$ and $n'<n$. Therefore, one can calculate all these numbers only knowing the initial condition $C_{g,1}(\mathbf{v})=1$ if $g=0$ and $v_i=1$ for all $i$, and $C_{g,1}(\mathbf{v})=0$ otherwise.
%\end{rem}

\subsection{Univariate generating series}

A relevant corollary of our results is that the formulas we obtain allow to compute the univariate generating series of some given models of maps ($2k$-angulations counted by faces, constellations counted by star vertices, etc.). To illustrate this fact, fix an integer $k$ and let $F_g(z)$ be the generating series of genus $g$ bipartite $2k$-angulations:
\[F_g(z)=\sum_{n>0} A^{(k)}_{g,n}z^n\]
with the coefficients $A^{(k)}_{g,n}$ as defined in Corollary \ref{thm_angulations}.
Our formula gives an algorithm to compute every $F_g$ for $g\geq 1$, given $F_0$.
Indeed, take $g\geq 1$, Corollary \ref{thm_angulations} rewrites
\begin{equation}\label{eq_series}
\Delta F_g=\phi(z,F_0,F_1,…,F_{g-1})
\end{equation}
with
\[\Delta=\binom{kD+1}{2}-\left(\binom{(k-1)D+2}{2}F_{0}\right)(kD+1)-\left((kD+1)F_{0}\right)\binom{(k-1)D+2}{2}-\binom{(k-1)D+2}{2},\]
where $D=z\frac{\partial}{\partial z}$, and $\phi$ is a polynomial in its variables and their (first and second) derivatives.
It is well known (see for instance \cite{BDG})) that
\[F_0=t-z\binom{2k-1}{k+1}t^{k+1}-1\]
with the change of variable
\[t=1+z\binom{2k-1}{k}t^k.\]
Note that we have a "$-1$" in the expression of $F_0$ because we do not count the "empty map".

Assuming we know $F_h$ for $h<g$, this gives a linear, second order ODE in $F_g$ (with respect to the variable $t$). Since all the $F_g$'s are rational in $t$ (see for instance \cite{ChapuyFang}), all the coefficients of the equation are themselves rational, and the solutions can be computed explicitly. The initial conditions are given by the two following facts: $[z^0]F_g=0$ and $[z^1]F_g$ is the number of unicellular bipartite maps of genus $g$ with $k$ edges, that can for instance be computed using Theorem \ref{thm_unicellular}.

\section{Monotone Hurwitz numbers}\label{sec:monotone}
In this section, we derive a recurrence formula for monotone Hurwitz numbers, in a similar fashion as in previous sections. These numbers, which appear in the calculation of the HCIZ integral, were introduced in \cite{monotone}.

\begin{defn}
For two transpositions of $\mathfrak{S}_n$, we say that $(i,j)\preceq (k,l)$ if $\max(i,j)\leq \max(k,l)$.
The double monotone Hurwitz number $\vec{H}_{g,n}^{\lambda,\mu}$ is $\frac{1}{n!}$ times the number of tuples $(t_1,t_2,…,t_r,\sigma_\lambda,\sigma_\mu)$ of permutations of $S_n$ such that:
\begin{itemize}
\item $r=l(\lambda)+l(\mu)+2g-2$ where $l(\lambda)$ is the number of parts of $\lambda$
\item $t_1,t_2,…,t_r$ is an increasing sequence of transpositions
\item $\sigma_\lambda$ (resp. $\sigma_\mu)$ has cycle type $\lambda$ (resp. $\mu$)
\item $t_1\cdot t_2 \cdot …\cdot t_{r}=\sigma_\lambda\sigma_\mu$
\item the permutations $t_1,t_2,…,t_r,\sigma_\lambda$ act transitively on ${1,2,…,n}$
\end{itemize}

The simple monotone Hurwitz numbers $\vec{H}_{g,n}
^{\lambda}$ are defined as $\vec{H}_{g,n}^{\lambda}=\vec{H}_{g,n}^{\lambda,1^n}$.

\end{defn}

We will set $W_{g,n}^{\lambda,\mu}$ to be the same numbers without the transitivity condition, and introduce 
\[\tau(z,\mathbf{p},\mathbf{q},u)=\sum_{\substack{n\geq 0\\ |\lambda |=|\mu |=n\\ r\geq 0}} \frac{z^n}{n!}p_\lambda q_\mu u^{r} W_{g,n}^{\lambda,\mu}\]
with $g$ such that $r=l(\lambda)+l(\mu)+2g-2$. $H=\log\tau$ is the generating function of the $\vec{H}_{g,n}^{\lambda,\mu}$.

As before, it can be shown (see for instance \cite{GPH}) that
\[\tau(z,\mathbf{p},\mathbf{q},u)=\mel{\emptyset}{\Gamma_+(\mathbf{p})z^H\Lambda \Gamma_-(\mathbf{q})}{\emptyset}\]
with $\Lambda=\exp(-F(-u))$, where $F$ is the function defined in Lemma \ref{lem_operator_form}.
A general equation similar to \eqref{eq_master} can be derived:
\begin{equation}\label{eq_master_monotone}
D  H_{1,1} - H_{1,1} = H_{1,1} \left( D H \left( \frac { z } { 1 + u } , \frac { u } { 1 + u } \right) + D H \left( \frac { z } { 1 - u } , \frac { u } { 1 - u } \right) - 2 D H \right)
\end{equation}
with $H_{1,1}=\frac { \partial ^ { 2 } } { \partial p _ { 1 } \partial q _ { 1 } } H$ and $D=z\frac{\partial}{\partial z}$.

Similarly as with constellations, in general we cannot even track the cycle type of $\sigma_\lambda$, although, from the specialization $p_i=q_i=\mathbb{1}_{i=1}$ for all $i$ we can obtain a recurrence formula for the unramified monotone Hurwitz numbers $\vec{H}_{g,n}=\vec{H}_{g,n}^{1^n}$:
\begin{equation}\label{eq_monotone}
n\binom{n}{2}\vec{H}_{g,n}=\sum_{\substack{n_1+n_2=n\\g^*\geq 0\\g_1+g_2+g^*=g}}n_1^2n_2\binom{3n_2+2g_2+2g^*-1}{2g^*+2}\vec{H}_{g_1,n_1}\vec{H}_{g_2,n_2}.
\end{equation}

\begin{rem}
In this paper, the number $\vec{H}_{g,n}^{\lambda,\mu}$ are defined with a scaling factor of $\frac{1}{n!}$ to make the formula simpler, this is a different convention as in \cite{monotone}. Formula \eqref{eq_monotone} allows to compute all the $\vec{H}_{g,n}$ only knowing that $\vec{H}_{0,1}=1$.
\end{rem}

\section*{Acknowledgements}
The author wishes to thank Guillaume Chapuy for suggesting the problem and for useful discussions, as well as anonymous reviewers for their comments.

\bibliography{sample.bib}{}

\begin{thebibliography}{10}

\bibitem{Adrianov}
N.~M. Adrianov.
\newblock An analogue of the {H}arer-{Z}agier formula for unicellular two-color
  maps.
\newblock {\em Funktsional. Anal. i Prilozhen.}, 31(3):1--9, 95, 1997.

\bibitem{AP}
M.~Albenque and D.~Poulalhon.
\newblock Generic method for bijections between blossoming trees and planar
  maps.
\newblock {\em Electron. J. Comb. vol.22, paper P2.38}, 2015.

\bibitem{ACEH}
A.~Alexandrov, G.~Chapuy, B.~Eynard, and J.~Harnad.
\newblock Fermionic approach to weighted {H}urwitz numbers and topological
  recursion.
\newblock {\em Comm. Math. Phys. 360,777-826}, 2018.

\bibitem{BC}
E.A. Bender and E.R. Canfield.
\newblock The asymptotic number of rooted maps on a surface.
\newblock {\em Journal of Combinatorial Theory, Series A}, 43(2):244--257,
  1986.

\bibitem{BF}
O.~Bernardi and E.~Fusy.
\newblock A bijection for triangulations, quadrangulations, pentagulations,
  etc.
\newblock {\em Journal of Combinatorial Theory, Series A 119, 1, 218-244},
  2012.

\bibitem{BDG}
J.~Bouttier, P.~Di~Francesco, and E.~Guitter.
\newblock Planar maps as labeled mobiles.
\newblock {\em Elec. Jour. of Combinatorics Vol 11 R69}, 2004.

\bibitem{BudzLouf}
Thomas Budzinski and Baptiste Louf.
\newblock Local limits of uniform triangulations in high genus, 2019.

\bibitem{CC}
S.~R. Carrell and G.~Chapuy.
\newblock Simple recurrence formulas to count maps on orientable surfaces.
\newblock {\em Journal of Combinatorial Theory, Series A, 133:58--75}, 2015.

\bibitem{Constellations}
G.~Chapuy.
\newblock Asymptotic enumeration of constellations and related families of maps
  on orientable surfaces.
\newblock {\em Comb. Probab. Comput.}, 18(4):477–516, 2009.

\bibitem{trisections}
G.~Chapuy.
\newblock A new combinatorial identity for unicellular maps, via a direct
  bijective approach.
\newblock {\em Adv. in Appl. Math.}, 47(4):874--893, 2011.

\bibitem{ChapuyFang}
G.~Chapuy and W.~Fang.
\newblock Generating functions of bipartite maps on orientable surfaces.
\newblock {\em Electron. J. Combin.}, 23(3):Paper 3.31, 37, 2016.

\bibitem{CFF}
G.~Chapuy, V.~F\'eray, and E.~Fusy.
\newblock A simple model of trees for unicellular maps.
\newblock {\em Journal of Combinatorial Theory, Series A 120, 8, Pages
  2064-2092}, 2013.

\bibitem{CMS}
G.~Chapuy, M.~Marcus, and G.~Schaeffer.
\newblock A bijection for rooted maps on orientable surfaces.
\newblock {\em SIAM Journal on Discrete Mathematics, 23(3):1587-1611}, 2009.

\bibitem{PSHT}
Nicolas Curien.
\newblock Planar stochastic hyperbolic triangulations.
\newblock {\em Probab. Theory Related Fields}, 165(3-4):509--540, 2016.

\bibitem{Zagier}
B~Dubrovin, Di~Yang, and D~Zagier.
\newblock Classical {H}urwitz numbers and related combinatorics.
\newblock {\em Moscow Mathematical Journal}, 17:601--633, 2017.

\bibitem{Eynard}
B.~Eynard.
\newblock {\em Counting Surfaces}.
\newblock Springer Basel, 2016.

\bibitem{Gao}
Z.~Gao.
\newblock The number of degree restricted maps on general surfaces.
\newblock {\em Discrete Mathematics}, 123(1-3):47--63, 1993.

\bibitem{monotone}
I.~P. Goulden, M.~Guay-Paquet, and J.~Novak.
\newblock Monotone {H}urwitz numbers and the {HCIZ} integral.
\newblock {\em Annales mathématiques Blaise Pascal}, 21(1):71--89, 2014.

\bibitem{GJ}
I.~P. Goulden and D.~M. Jackson.
\newblock The {KP} hierarchy, branched covers, and triangulations.
\newblock {\em Advances in Mathematics 219}, 2008.

\bibitem{GPH}
M.~Guay-Paquet and J.~Harnad.
\newblock 2{D} {T}oda {$\tau$}-functions as combinatorial generating functions.
\newblock {\em Lett. Math. Phys.}, 105(6):827--852, 2015.

\bibitem{KZ}
M.~Kazarian and P.~Zograf.
\newblock Virasoro constraints and topological recursion for {G}rothendieck's
  dessin counting.
\newblock {\em Lett. Math. Phys. 105 (8), 1057-1084}, 2015.

\bibitem{Lep}
M.~Lepoutre.
\newblock Blossoming bijection for higher-genus maps.
\newblock {\em Journal of Combinatorial Theory, Series A}, 165:187 -- 224,
  2019.

\bibitem{Louf}
B.~Louf.
\newblock A new family of bijections for planar maps.
\newblock {\em Journal of Combinatorial Theory, Series A}, 168:374 -- 395,
  2019.

\bibitem{Solitons}
T.~Miwa, M.~Jimbo, and E.~Date.
\newblock {\em Solitons : differential equations, symmetries, and infinite
  dimensional algebras}.
\newblock CUP, 2000.

\bibitem{Ok}
A.~Okounkov.
\newblock Toda equations for {H}urwitz numbers.
\newblock {\em Math. Res. Lett.}, 7, 2000.

\bibitem{Ok2}
A.~Okounkov.
\newblock Infinite wedge and random partitions.
\newblock {\em Sel. Math. New Ser.}, 2001.

\bibitem{Pandharipande}
R.~Pandharipande.
\newblock The {T}oda equations and the {G}romov--{W}itten theory of the riemann
  sphere.
\newblock {\em Letters in Mathematical Physics}, 53(1):59--74, 2000.

\bibitem{PS}
D.~Poulalhon and G.~Schaeffer.
\newblock Factorizations of large cycles in the symmetric group.
\newblock {\em Discrete Math.}, 254(1-3):433--458, 2002.

\bibitem{theseSch}
G.~Schaeffer.
\newblock {\em Conjugaison d'arbres et cartes combinatoires al\'eatoires}.
\newblock Th\`ese de doctorat, Universit\'e Bordeaux I, 1998.

\bibitem{Tutte}
W.~T. Tutte.
\newblock A census of planar maps.
\newblock {\em Can. J. Math.}, 15(0):249--271, jan 1963.

\bibitem{LW}
T.R.S Walsh and A.B Lehman.
\newblock Counting rooted maps by genus. i.
\newblock {\em Journal of Combinatorial Theory, Series B}, 13(3):192--218,
  1972.

\end{thebibliography}
\bibliographystyle{plain}

%\vspace{0.5cm}

%\vspace{1cm}

%\newpage
%%\vspace{0.1cm}
%%

%\tableofcontents

 \end{document}